\documentclass[11pt]{amsart}
\usepackage{amsopn, amssymb, amscd, amsmath, amsthm}
\usepackage{hyperref}
\usepackage{graphicx, graphics, epsfig}
\usepackage{enumerate}

\usepackage{todonotes} %para las notas al margen. Para sacarlas, comentar esto con %%%  y descomentar la linea siguiente

\usepackage{lscape, pdflscape}  %para escribir en vertical
\usepackage{array, multirow} %multirows en la tabla
\usepackage{afterpage, capt-of} %otras cosas para la tabla

\allowdisplaybreaks

\newcommand{\nc}{\newcommand}

\nc{\Uca}{\mathcal{U}}

\nc{\RS}{\sigma_{\operatorname{Ric}}} \nc{\REV}{\operatorname{RicEV}}

\nc{\pr}{\operatorname{pr}}

\nc{\Gv}{{\G_{(\vg_i)} } }   \nc{\ggov}{{\ggo_{(\vg_i)} }}   \nc{\Gvt}{{\G^t_{(\vg_i)} } } 
\nc{\ig}{\mathfrak{i}}

\nc{\Gl}{\mathsf{GL}} \nc{\Or}{\mathsf{O}}  \nc{\SO}{\mathsf{SO}}   \nc{\Sl}{\mathsf{SL}}
\nc{\G}{\mathsf{G}} \nc{\K}{\mathsf{K}}  \nc{\T}{\mathsf{T}} \nc{\Lsf}{\mathsf{L}}
\nc{\Qb}{\mathsf{Q}_\Beta} \nc{\Hb}{\mathsf{H}_\Beta} \nc{\Ub}{\mathsf{U}_\Beta}
\nc{\Gb}{\mathsf{G}_\Beta} \nc{\Kb}{\mathsf{K}_\Beta}
\nc{\PPP}{\mathsf{P}} \nc{\U}{\mathsf{U}} \nc{\N}{\mathsf{N}} \nc{\Ss}{\mathsf{S}} \nc{\Aa}{\mathsf{A}}
\nc{\Hh}{\mathsf{H}}

\nc{\la}{\langle} \nc{\ra}{\rangle}

\nc{\laH}{\la\!\la} \nc{\raH}{\ra\!\ra}

\nc{\ipH}{{\laH \cdot, \cdot \raH}}

\nc{\iph}{{\la h \cdot , h \cdot \ra}}

\nc{\Vg}{{V(\ggo)}}

\nc{\alert}{\color{blue}}

\nc{\fg}{\mathfrak{f}}  \nc{\vg}{\mathfrak{v}} \nc{\wg}{\mathfrak{w}} \nc{\zg}{\mathfrak{z}} \nc{\ngo}{\mathfrak{n}} \nc{\kg}{\mathfrak{k}} \nc{\mg}{\mathfrak{m}} \nc{\bg}{\mathfrak{b}} \nc{\ggo}{\mathfrak{g}} \nc{\ggob}{\overline{\mathfrak{g}}} \nc{\sog}{\mathfrak{so}} \nc{\sug}{\mathfrak{su}} \nc{\spg}{\mathfrak{sp}} \nc{\slg}{\mathfrak{sl}} \nc{\glg}{\mathfrak{gl}} \nc{\cg}{\mathfrak{c}} \nc{\rg}{\mathfrak{r}}  \nc{\hg}{\mathfrak{h}} \nc{\tgo}{\mathfrak{t}} \nc{\ug}{\mathfrak{u}} \nc{\dg}{\mathfrak{d}} \nc{\ag}{\mathfrak{a}} \nc{\pg}{\mathfrak{p}} \nc{\sg}{\mathfrak{s}} \nc{\affg}{\mathfrak{aff}} \nc{\qg}{\mathfrak{q}}
\nc{\Xg}{\mathfrak{X}} \nc{\lgo}{\mathfrak{l}}

\nc{\pca}{\mathcal{P}} \nc{\nca}{\mathcal{N}} \nc{\lca}{\mathcal{L}} \nc{\oca}{\mathcal{O}} \nc{\mca}{\mathcal{M}} \nc{\tca}{\mathcal{T}} \nc{\aca}{\mathcal{A}} \nc{\cca}{\mathcal{C}} \nc{\gca}{\mathcal{G}} \nc{\sca}{\mathcal{S}} \nc{\hca}{\mathcal{H}} \nc{\bca}{\mathcal{B}} \nc{\dca}{\mathcal{D}}

\nc{\vp}{\varphi} \nc{\ddt}{\tfrac{{\rm d}}{{\rm d}t}} \nc{\dds}{\tfrac{{\rm d}}{{\rm d}s}} \nc{\ddtbig}{\frac{{\rm d}}{{\rm d}t}} \nc{\dd}{{\rm d}}
\nc{\dpar}{\tfrac{\partial}{\partial t}} \nc{\im}{\mathtt{i}}

 \nc{\SU}{\mathsf{SU}} 
\nc{\RR}{{\mathbb R}} \nc{\HH}{{\mathbb H}} \nc{\CC}{{\mathbb C}} \nc{\ZZ}{{\mathbb Z}}
\nc{\FF}{{\mathbb F}} \nc{\NN}{{\mathbb N}} \nc{\QQ}{{\mathbb Q}} \nc{\PP}{{\mathbb P}}

\nc{\vs}{\vspace{.2cm}} \nc{\vsp}{\vspace{1cm}} \nc{\ip}{{\langle\cdot,\cdot\rangle}}
\nc{\ipp}{(\cdot,\cdot)} \nc{\unm}{\tfrac{1}{2}}
\nc{\unc}{\tfrac{1}{4}} \nc{\und}{\tfrac{1}{16}} \nc{\no}{\vs\noindent}
\nc{\lam}{\Lambda^2(\RR^n)^*\otimes\RR^n} \nc{\tangz}{{\rm T}^{\rm Zar}}
\nc{\lamg}{\Lambda^2\ggo^*\otimes\ggo}
\nc{\nor}{{\sf n}}  \nc{\mum}{/\!\!/} \nc{\kir}{/\!\!/\!\!/}
\nc{\Ri}{\tfrac{4\Ric_{\mu}}{||\mu||^2}} \nc{\ds}{\displaystyle}
\nc{\lb}{[\cdot,\cdot]} \nc{\isn}{\tfrac{1}{||v||^2}}
\nc{\gkp}{(\ggo=\kg\oplus\pg,\ip)} \nc{\ukh}{(\ug=\kg\oplus\hg,\ip)}
\nc{\tgkp}{(\tilde{\ggo}=\kg\oplus\pg,\ip)}
\nc{\wt}{\widetilde}
\nc{\raw}{\rightarrow} \nc{\lraw}{\longrightarrow} \nc{\hqn}{\mathcal{H}_{q,n}}

\nc{\Spec}{\operatorname{Spec}}

\nc{\ad}{\operatorname{ad}}  \nc{\Aut}{\operatorname{Aut}}   \nc{\Inn}{\operatorname{Inn}}   \nc{\Lie}{\operatorname{Lie}} \nc{\Ad}{\operatorname{Ad}} \nc{\Der}{\operatorname{Der}} \nc{\rad}{\operatorname{rad}} \nc{\kf}{\operatorname{B}}

\nc{\End}{\operatorname{End}} \nc{\rank}{\operatorname{rank}} \nc{\Ker}{\operatorname{Ker}} \nc{\tr}{\operatorname{tr}} \nc{\sym}{\operatorname{sym}} \nc{\diag}{\operatorname{diag}} \nc{\proy}{\operatorname{pr}} \nc{\Adj}{\operatorname{Adj}} \nc{\vspan}{\operatorname{span}}

\nc{\Hess}{\operatorname{Hess}}  \nc{\dif}{\operatorname{d}} \nc{\sen}{\operatorname{sen}} \nc{\grad}{\operatorname{grad}} \nc{\Order}{\operatorname{O}} \nc{\divg}{\operatorname{div}}

\nc{\Iso}{\operatorname{Iso}} \nc{\Diff}{\operatorname{Diff}} \nc{\Rc}{\operatorname{Rc}} \nc{\Ricci}{\operatorname{Ric}} \nc{\Riem}{\operatorname{Rm}} \nc{\scalar}{\operatorname{sc}} \nc{\scalarm}{\hat{\operatorname{R}}} \nc{\Riccim}{\widehat{\operatorname{Ric}}} \nc{\tang}{\operatorname{T}} \nc{\vol}{\operatorname{vol}}

\nc{\mm}{\operatorname{M}} \nc{\CH}{\operatorname{CH}} \nc{\Irr}{\operatorname{Irr}} \nc{\mcc}{\operatorname{mcc}} \nc{\m}{\operatorname{m}}

\nc{\Id}{\operatorname{Id}}  \nc{\mmm}{\operatorname{m}}

\theoremstyle{plain}
\newtheorem{theorem}{Theorem}[section]
\newtheorem{proposition}[theorem]{Proposition}
\newtheorem{corollary}[theorem]{Corollary}
\newtheorem{lemma}[theorem]{Lemma}

\theoremstyle{definition}

\newtheorem{problem}[theorem]{Problem}

\theoremstyle{plain}

\theoremstyle{remark}
\newtheorem{remark}[theorem]{Remark}
\newtheorem{notation}[theorem]{Notation}
\newtheorem{example}[theorem]{Example}

\title[SKT structures on nilmanifolds]{SKT structures on nilmanifolds}

\author{Romina M.~Arroyo}
\author{Marina Nicolini}
\thanks{
 }

\begin{document}
\begin{abstract} 
The aim of this article is to study the existence of invariant SKT structures on nilmanifolds. More precisely, we  give a negative answer to the question of whether there exist a $k$-step ($k>2$) complex nilmanifold admitting an invariant SKT metric. We also provide a construction which serves as a tool to generate examples of invariant SKT structures on $2$-step nilmanifolds in arbitrary dimensions.  
\end{abstract}

\maketitle

% \tableofcontents
%\vspace{-13pt}

\section{Introduction}

Let $(M,J,g)$ be a Hermitian manifold with associated fundamental form $\omega$. If $\omega$ is not closed, it means that the manifold is not Kähler, then the Levi-Civita connection does not preserve the complex structure. There are plenty of connections preserving both structures (\cite{Gau97}), but there is only one such that the torsion $3$-form  is totally skew-symmetric, the so-called \emph{Bismut connection}.  When the $3$-torsion form is in addition closed, the Hermitian manifold $(M,J,g)$ is said to be \emph{strong Kähler with torsion} (SKT for short) or \emph{pluriclosed}.

We are interested in the study of invariant SKT structures on nilmanifolds. Here, $M$ is a compact quotient $\Gamma \backslash N$, of a simply-connected nilpotent Lie group $N$ by a co-compact lattice $\Gamma$, and the Hermitian structure comes
from a left-invariant Hermitian structure on the Lie group~ $N$. 

Over recent years, invariant SKT structures on nilmanifolds have been studied by many authors, and remarkably, still not much is known about their existence. The classification in dimensions $4$, $6$ and $8$ was obtained in \cite{MadsenSwann,FinoPartonSalamon,EFV12}, respectively. Regarding higher dimensions, a characterization of a class of SKT nilmanifolds was studied in \cite{ZZ19}, where the complex structure is nilpotent and the compatible metric is Kähler-like. On the other hand, to the best of our knowledge, the only non-existence results in arbitrary dimensions are given in \cite{EFV12}.

All known examples in the literature of nilmanifolds admitting an SKT structure are $2$-step nilpotent. In \cite[Theorem 1.2]{EFV12}, it is stated that the latter exhaust all the nilpotent examples. Unfortunately, its proof has a gap (see \cite{FV2019}), leading to the following problem.

\begin{problem}\label{problem}(\cite{FV2019,FS20,DFFL21,FTV21})
Does a $k$-step complex nilmanifold ($k>2$) admitting an invariant SKT structure exist?
\end{problem}

The first partial negative answer to Problem \ref{problem} was given in \cite{ZZ19}, where the authors work on Kähler-like structures on nilmanifolds asumming nilpotency on the complex structure. The latter turn out to be $2$-step nilmanifolds and the complex structure is necessarily abelian. After that, in the recent work \cite{FTV21}, a negative answer to the problem was obtained on complex nilmanifolds with the abelian assumption in the complex structure.

Our main result gives a complete answer to Problem \ref{problem}. As an important consequence, \cite[Theorem 2.3]{Enr13}, \cite[Theorem 1.1]{EFV12} and \cite[Theorem 1.1]{FV16} turn out to be valid. Moreover, the long-time behaviour of the pluriclosed flow of invariant SKT structures on nilmanifolds is now completely understood (see \cite[Theorem A]{AL19}). 

\begin{theorem}\label{main theorem}
Any nilmanifold admitting an invariant SKT structure is either a torus or $2$-step nilpotent.
\end{theorem}

According to Theorem \ref{main theorem}, the next move is to understand invariant SKT structures on $2$-step nilmanifolds. The really hard problem is to reach new examples in higher dimensions, and the lack of them motivated us to develope a method to construct families of invariant SKT structures on nilmanifolds in higher dimensions starting with low dimensional ones (see Section \ref{contructionofexamples}). This machinery provides explicit examples in every complex dimension. Moreover, as far as we know, we give the first examples of SKT structures on nilmanifolds with non-abelian complex structure in higher dimensions.

We now give some insight into our main results. Any invariant SKT structure on a nilmanifold is determined by the following infinitesimal data, which we call an \emph{SKT Lie algebra}: a nilpotent Lie algebra $\ggo$, a complex structure $J$ on $\ggo$ and an inner product on $\ggo$ satisfying a system of equations on $\ggo$ due to the SKT condition.  The key idea in the proof of Theorem \ref{main theorem} is to write $\ggo=\vspan\{ e_1,e_2\} \oplus \ngo$, as the orthogonal sum of a subspace and an ideal $\ngo$, where both spaces are $J$-invariant (see \cite[Corollary 1.4]{Salamon01}), and to prove that $(\ngo,J|_{\ngo},\ip|_{\ngo})$ is also SKT (see Section \ref{sec-conj}). Then, $\ggo$ is determined by   
 
\begin{equation*}
A:=\ad(e_1)_{\ngo}, \quad B:=\ad(e_2)_{\ngo}, \quad X:=[e_1,e_2],\quad \mbox{and} \quad [\cdot,\cdot]_{\ngo}, 
\end{equation*} where $\ad(e_i)_{\ngo}$ denotes the projection of $\ad(e_1)|_{\ngo}$ onto $\ngo$, for $i=1,2$. If we apply induction on $n$ to $\dim\ggo = 2n$, then by induction hypothesis, the ideal $\ngo$ is forced to be abelian or $2$-step nilpotent (see Section \ref{sec-ABX}). Therefore, $\ngo$ can be decomposed as $\ngo=\vg \oplus \zg,$ 
where  $\zg$ is the center of $\ngo$ and $\vg:=\zg^{\perp}$ ($\ngo=\zg$ when $\ngo$ is abelian). Since $A,B\in\Der(\ngo)$, then
\[
A=\left[\begin{matrix}A_\vg & 0\\ \ast & A_\zg \end{matrix}\right], \qquad B=\left[\begin{matrix}B_\vg & 0\\ \ast & B_\zg \end{matrix}\right].
\] 
We first show that $A_\zg=0$ and $B_\zg=0$ by using the SKT condition (see Corollary \ref{A=0-abelian} for an abelian $\ngo$ and Lemma \ref{A3=0} for a $2$-step nilpotent $\ngo$). This fact together with the nilpotency of $\ggo$ and the integrability of $J$ are the ingredients to demonstrate that $A_\vg=0$, $B_\vg=0$ and $X\in\zg$, which proves that $\ggo$ is at most $2$-step nilpotent.

Our second main result is a method that provides new explicit examples of SKT Lie algebras (see Section \ref{contructionofexamples}). We start with two SKT $2$-step nilpotent Lie algebras $(\ngo_1,J_1,\ip_1)$ and $(\ngo_2,J_2,\ip_2)$  of dimensions $n_1$ and $n_2$, respectively, satisfying 
\[
\ngo_i=\vg_i\oplus \zg_i\quad \mbox{and} \quad \dim{\zg_i}>\dim{[\ngo_i,\ngo_i]}, \quad i=1,2,
\] where $\zg_i$ is the center of $\ngo_i$ and $\vg_i:=\zg_i^{\perp}$, $i=1,2$, and we construct a new SKT Lie algebra of dimension $n_1+n_2+2$ by setting $\ggo= \ngo_1\oplus\ngo_2\oplus \la Z, W\ra$ with Lie bracket given by 
\begin{equation*}
\lb|_{\ngo_1\times\ngo_1}=\lb_{\ngo_1},\quad \lb|_{\ngo_2\times\ngo_2}=\lb_{\ngo_2},\quad [Z,W]=X_{n_1}+Y_{n_2},    
\end{equation*} 
where $X_{n_1}\in\zg_1\cap[\ngo_1,\ngo_1]^\perp$ and $Y_{n_2}\in\zg_2\cap[\ngo_2,\ngo_2]^\perp$. The complex structure is defined as
$$J=\left[\begin{smallmatrix}
\\J_1&&&\\&J_2&&\\&&0&-1\\&&1&0
\end{smallmatrix}\right],$$ and the inner product is the one that makes the above decomposition of $\ggo$ orthogonal while extending $\ip_1$ and $\ip_2$. The SKT Lie algebra $(\ggo,J,\ip)$ is irreducible, in the sense that it is not a product of two SKT Lie algebras, despite $\ggo$ is decomposable (see Section \ref{construction}). %The proof is a simple computation involving the SKT conditions of $(\ngo_1,J_1,\ip_1)$ and $(\ngo_2,J_2,\ip_2)$.

The organization of this article is as follows. In Section \ref{preliminaries} we review some basic facts about left-invariant SKT structures on Lie groups. In Section \ref{SKT nil} we prove some useful results. Then, we apply these results in Section \ref{sec-conj}, which is devoted to the proof of Theorem \ref{main theorem}. Finally, we present a construction in Section \ref{contructionofexamples} and explicit examples are provided. 

\vs \noindent {\it Acknowledgements.} We would like to thank Jorge Lauret and Ramiro A.\ Lafuente for fruitful discussions and useful comments on a first draft of this article. We are also grateful to Anna Fino for her careful reading of the paper. 

\section{Preliminaries}\label{preliminaries}

Given $(M^{2n},J)$ a differentiable manifold of real dimension $2n$ endowed with a complex structure, a Riemannian metric $g$ on $M$ is said to be {\it Hermitian} if $g(J\cdot, J\cdot) = g(\cdot, \cdot).$ The pair $(J,g)$ is called a {\it Hermitian structure} and $\omega(\cdot, \cdot)=g(J\cdot, \cdot)$ is the fundamental $2$-form associated to the pair. The {\it Bismut (or Strominger) connection} $\nabla^B$ on $M$ is the unique Hermitian connection (that is, $J$ and $g$ are parallel) with totally skew-symmetric torsion. That is, the tensor 
\begin{equation}\label{def-c}
    c(U,Y,Z) := g(U, T^B(Y,Z))
\end{equation} is a 3-form, where $T^B(Y,Z) = \nabla^B_Y Z - \nabla^B_Y Z - [Y,Z]$ is the torsion of $\nabla^B$ (see \cite{S86,Bis89}). The metric $g$ (or $\omega$) is called {\it strong K\"ahler with torsion (SKT)} or {\it pluriclosed} if its fundamental $2$-form satisfies $\partial \bar \partial \omega = 0$, or equivalently, the $3$-form $c$ is closed. In this case, $(J,g)$ is called a {\it SKT-structure} and the triple $(M,J,g)$ is said to be SKT.

We are interested in the study of {\it invariant SKT-structures} on Lie groups. Here, the universal cover $\tilde M$ of $M$ is diffeomeophic to a simply-connected Lie group $G$  and $\pi^* J$ and $\pi^*g$ are left-invariant tensors defining a Hermitian structure on $G$, where $\pi: G \to M$ denotes the universal covering map.

\subsection{Nilpotent Lie groups and Lie algebras}

Given a Lie group $G$ with Lie algebra $(\ggo,[\cdot,\cdot])$, for each $X\in\ggo$ we define the \emph{adjoint map} as the linear map $\ad(X):\ggo \to \ggo$, given by $\ad(X)(Y)= [X,Y]$ and we denote by $\zg(\ggo)$ the \emph{center} of $\ggo$, that is, $\zg(\ggo)=\{X \in \ggo \mid \ad(X) = 0\}$. 

For a Lie algebra $(\ggo,[\cdot,\cdot])$, we define its \emph{descending central series} by:
\begin{equation*}
    \ggo_0=\ggo, \quad \ggo_i=[\ggo,\ggo_{i-1}], \mbox{ for } i\geq 1.
\end{equation*}
A Lie algebra $\ggo$ is called \emph{nilpotent} if there exists $k \in \NN$ such that $\ggo_k=0$. In addition, if $\ggo_k=0$ and $\ggo_{k-1}\neq 0$, the Lie algebra is said to be \emph{$k$-step nilpotent}. A Lie group $G$ is \emph{($k$-step) nilpotent} if its Lie algebra is ($k$-step) nilpotent.

From now on, we simply denote by $\ggo$ the Lie algebra $(\ggo,[\cdot,\cdot])$.

\subsection{Hermitian structures on Lie groups}
Left-invariant Hermitian structures on simply-connected Lie groups $(G,J,g)$ are completely determined by $(\ggo,J(e),g(e))$, where $e$ is the identity of $G$. Here, if we denote by $J:=J(e)$ and $\ip:=g(e)$, then $J$ is a linear endomorphism $ J: \ggo \to \ggo$ satisfying $J^2 = -\Id_\ggo$ and the integrability condition
\[
[J\cdot,J\cdot]=[\cdot,\cdot]+J [J\cdot,\cdot]+J [\cdot,J\cdot],
\] and $\ip : \ggo \times \ggo \to \RR$ is an inner product on $\ggo$ such that $\la J \cdot , J \cdot \ra = \la \cdot ,\cdot \ra$. 
 
From now on, we denote the Hermitian manifold $(G,J,g)$ by $(\ggo,J,\ip).$

\subsection{SKT metrics on Lie groups}
The torsion $3$-form of the Bismut connection of a left-invariant Hermitian manifold $(\ggo,J,\ip)$ can be computed by (see \cite[(3.2)]{EFV12})
\begin{equation}\label{eqn_c}
  c(U,Y,Z) = -\la[JU, JY], Z\ra - \la[JY,JZ], U\ra - \la[JZ,JU], Y\ra, \qquad U, Y, Z \in \ggo,
\end{equation}
and its exterior derivative is thus given by

\begin{align}
  dc(W,U,Y,Z) =& \quad\,\la[J[W,U], JY], Z\ra                 + \la[JY,JZ], [W,U]\ra +      \la[JZ,J[W,U]],                 Y\ra\nonumber\\
                & - \la[J[W,Y], JU], Z\ra - \la[JU,JZ], [W,Y]\ra - \la[JZ,J[W,Y]], U\ra\nonumber\\
                & + \la[J[W,Z], JU],Y\ra + \la[JU,JY],[W,Z]\ra + \la[JY,J[W,Z]],U\ra\label{eqn_dc} \\
                & +\la[J[U,Y], JW], Z\ra + \la[JW,JZ],[U,Y]\ra + \la[JZ,J[U,Y]],W\ra\nonumber\\
                & - \la[J[U,Z], JW],Y\ra - \la[JW,JY],[U,Z]\ra - \la[JY,J[U,Z]], W\ra\nonumber\\
                & + \la[J[Y,Z], JW], U\ra + \la[JW,JU],[Y,Z]\ra + \la[JU,J[Y,Z]],W\ra .\nonumber \end{align}

Then, the SKT condition $dc = 0$ can be written as a system of equations on $\ggo$ involving the Lie bracket, the complex structure and the inner product.

From now on, we will say that $(\ggo,J,\ip)$ is SKT or an SKT Lie algebra if it is a Hermitian manifold satisfying \eqref{eqn_dc}.

\section{SKT nilmanifolds}\label{SKT nil}

The aim of this section is to prove two helpful results for the following sections. We also recall a result from \cite{EFV12} and set up some notation.   

\begin{proposition}\cite[Proposition 3.1]{EFV12} \label{centro-Jinv}
If $(\ggo,J,\ip)$ is SKT with $\ggo$  nilpotent, then $\zg(\ggo)$ is $J$-invariant. 
\end{proposition}

\begin{notation} \label{projection}
Let $V$ be a vector space. If $T\in\glg(V)$ and $W$ is a subspace of $V$, then $T_W$ denotes the projection of $T|_W$ onto $W$.
\end{notation} 

\begin{proposition}\label{ideal}
If $(\ggo,J,\ip)$ is SKT and $\ngo$ is a  $J$-invariant ideal of $\ggo$ of co-dimension $2$, then $(\ngo,J_\ngo,\ip|_\ngo)$ is SKT. 
%(no usa n 2-pasos)
\end{proposition}

\begin{proof}
Here and subsequently, $d_\ngo$ stands for the exterior derivative of the Lie algebra $(\ngo,[\cdot,\cdot]_\ngo)$. We first compute $d\alpha$ for $\alpha\in\Lambda^1\ngo^*$. To do this, we take $\{e_1,e_2\}$ in the orthogonal complement of $\ngo$ in $\ggo$, and $\{e^1,e^2\}$ the respective dual $1$-forms. Without loss of generality, we can assume that $Je_1=e_2$. 
For simplicity of notation, let $A$, $B$ and $X$ stand for $\ad(e_1)_\ngo$, $\ad(e_2)_\ngo$ and the projection of $[e_1,e_2]$ onto $\ngo$, respectively. It is straightforward to prove that
\begin{equation}\label{deriv}
d \alpha = e^1\wedge \theta(A)\alpha +e^2\wedge \theta(B)\alpha - \alpha(X) e^{12}  + d_\ngo\alpha, \quad\forall \alpha\in\Lambda^1\ngo^*,
\end{equation}
where $\theta:\glg(\ngo)\longrightarrow\End(\Lambda^k\ngo^*)$ denotes the representation obtained as the derivative of the natural left $\Gl(\ngo)$-action on each $\Lambda^k\ngo^*$, which is given by,
$$
\theta(B)\gamma = \ddt\Big|_0 e^{tB}\cdot\gamma = -\left(\gamma(B\cdot,\dots,\cdot) + \dots +\gamma(\cdot,\dots,B\cdot)\right), \qquad\forall\gamma\in\Lambda^k\ngo^*,\;B\in\glg(\ngo).
$$
Equation \eqref{deriv} can be generalized to obtain the exterior derivative of any $k$-form in $\ngo^*$. On the other hand, from \eqref{eqn_dc}, it is clear that there exist $\alpha,\beta\in\Lambda^2\ngo^*$ and $\gamma\in\Lambda^1\ngo^*$ such that
\[c = \alpha \wedge e^1 + \beta\wedge e^2 + \gamma\wedge e^{12} +c_\ngo,     \]
where $c_\ngo$ is the torsion $3$-form of $\ngo$ (see \eqref{def-c}).
We claim that $0=d_\ngo c_\ngo\in\Lambda^4\ngo^*$. Indeed, from \eqref{deriv}, we obtain that
\begin{align*}
    0=&dc = d\alpha \wedge e^1 + d\beta\wedge e^2 + d\gamma\wedge e^{12} +dc_\ngo\\
    =& (-\theta(B)\alpha+ \theta(A)\beta + d_\ngo\gamma+\delta)\wedge e^{12}+(d_\ngo\alpha -\theta(A)c_\ngo)\wedge e^1 +(d_\ngo\beta-\theta(B)c_\ngo)\wedge e^2+d_\ngo c_\ngo,
\end{align*}
where $\delta\in\Lambda^2\ngo^*$ is such that $dc_\ngo=e^1\wedge\theta(A)c_\ngo+e^2\wedge\theta(B)c_\ngo+\delta\wedge e^{12}+d_\ngo c_\ngo$. 
This equation yields $d_\ngo c_\ngo=0$, and therefore $(\ngo,J_\ngo,\ip_\ngo)$ is SKT. 
\end{proof}

\begin{remark}\label{ideal-coro}
Note that the $2$-codimensional hypotesis on $\ngo$ can be removed. The proof follows in the same way, but the notation can get tricky. 
\end{remark}

\begin{lemma}\label{centro} Let $(\ngo,J,\ip)$ be an SKT Lie algebra where $\ngo$ is $2$-step nilpotent. Then, $Y\in\zg(\ngo)$ if and only if $[Y,JY]=0$. 
\end{lemma}
\begin{proof}
Let $W,Y \in \ngo$, then by (\ref{eqn_dc}),
\begin{equation*}\label{eqn_dc_WZ}
\begin{array}{lcl}
  dc(W,JW,Y,JY)   &=& + \la[J[W,JW], JY], JY\ra - \la[JY,Y], [W,JW]\ra - \la[Y,J[W,JW]], Y\ra\\
                && + \la[J[W,Y], W], JY\ra - \la[W,Y], [W,Y]\ra + \la[Y,J[W,Y]], JW\ra\\
                && - \la[J[W,JY], W],Y\ra - \la[W,JY],[W,JY]\ra + \la[JY,J[W,JY]],JW\ra \\
                && + \la[J[JW,Y], JW], JY\ra - \la[JW,Y],[JW,Y]\ra - \la[Y,J[JW,Y]],W\ra\\
                && - \la[J[JW,JY], JW],Y\ra - \la[JW,JY],[JW,JY]\ra - \la[JY,J[JW,JY]], W\ra\\
                && + \la[J[Y,JY], JW], JW\ra - \la[JW,W],[Y,JY]\ra - \la[W,J[Y,JY]],W\ra . 
\end{array}
\end{equation*}
Since $[\ngo,\ngo]\subseteq \zg(\ngo)$ and $\zg(\ngo)$ is $J$-invariant by Proposition \ref{centro-Jinv}, we have that  

\begin{equation*}
\begin{array}{lcl}
 0=dc(W,JW,Y,JY)   &=&  - \la[JY,Y], [W,JW]\ra  - \la[W,Y], [W,Y]\ra  - \la[W,JY],[W,JY]\ra \\
                &&  - \la[JW,Y],[JW,Y]\ra  - \la[JW,JY],[JW,JY]\ra - \la[JW,W],[Y,JY]\ra  . 
\end{array}
\end{equation*}
That means
\[
-2 \la[Y,JY], [W,JW]\ra = - \|[W,Y]\|^2 - \|[W,JY]\|^2 - \|[JW,Y]\|^2 - \|[JW,JY]\|^2.
\]
Therefore, $[Y,JY]=0$ if and only if $[W,Y]=0$ for all $W \in \ngo$, that is $Y \in \zg(\ngo).$ 
\end{proof}

\section{Proof of Theorem \ref{main theorem}}\label{sec-conj}
Let $(\ggo,J,\ip)$ be a $2n$-dimensional real nilpotent Lie algebra endowed with a Hermitian structure. Using \cite[Corollary 1.4]{Salamon01}, there exists an orthonormal basis $\{e^1,\cdots,e^{2n}\}$ of $\ggo^*$ satisfying that $Je^1=e^2$ and
\begin{equation}\label{eq_bracket}
d e^i \in \sum_{j,k<i} c_{jk}^i e^{jk},
\end{equation} where $-c_{jk}^i$ denote the structural constants of the Lie bracket on $\ggo$. In particular, $de^1=0$, $de^2=0$ and $\ngo:=\vspan\{ e_1,e_2\}^{\perp}$ is a $J$-invariant ideal of $\ggo$.
Then, the Lie bracket of $\ggo$ is determined by
\begin{equation}\label{structure}
A:=\ad(e_1)_{\ngo}, \quad B:=\ad(e_2)_{\ngo}, \quad X:=[e_1,e_2],\quad \mbox{and} \quad [\cdot,\cdot]_{\ngo}. 
\end{equation} 
In particular,  $(\ngo,J_\ngo,\ip|_\ngo)$ is a $(2n-2)$-dimensional real nilpotent Lie algebra endowed with a Hermitian structure. Moreover, if $(\ggo, J, \ip)$ is SKT,  then $(\ngo,J_\ngo,\ip|_\ngo)$ turns out to be SKT by Proposition \ref{ideal}.

\begin{remark}\label{ABJ}
The integrability condition implies that
\[
   [Je_1,JY]=[e_1,Y]+J[Je_1,Y]+J[e_1,JY],\;\forall 
   Y\in\ngo, \mbox{ i.e. } [J,A]=J[B,J].
\]
\end{remark}

From now on, we will denote by $(\ggo_{A,B,X,\ngo},J,\ip)$ the Hermitian manifold such that $\ggo_{A,B,X,\ngo}$ is the nilpotent Lie algebra defined as in \eqref{structure}, $Je_1=e_2$, $J\ngo\subseteq\ngo$, and $\ip$ satisfies that $\la e_1,e_2\ra=0$ and $\{e_1,e_2\}\perp\ngo$.

\subsection{$2$-step  nilpotent ideal of codimension $2$ }\label{sec-ABX}

The aim of this section is to prove Theorem \ref{main theorem} for $(\ggo_{A,B,X,\ngo},J,\ip)$ in the case that $\ngo$ is $2$-step nilpotent. 

Assume that $(\ggo_{A,B,X,\ngo},J,\ip)$ is an SKT Lie algebra and $\ngo$ is $2$-step nilpotent. By Proposition \ref{ideal},  $(\ngo,J_\ngo,\ip|_\ngo)$ is SKT. Hence, we can decompose $\ngo$ as 
\[
\ngo=\vg \oplus \zg,
\] 
where  $\zg:=\zg(\ngo)$ is the center of $\ngo$ and $\vg:=\zg^{\perp}$. Note that $\zg$ and $\vg$ are invariant by $J$ and $[\vg,\vg]_\ngo\subseteq\zg$. According to the above decomposition, $J_{\ngo}$ is determined by $J_{\vg}$ and $J_{\zg}$ (see Notation \ref{projection}). In addition, since $A,B\in\Der(\ngo)$, then
\[
A=\left[\begin{matrix}A_\vg & 0\\ \ast & A_\zg \end{matrix}\right], \qquad B=\left[\begin{matrix}B_\vg & 0\\ \ast & B_\zg \end{matrix}\right].
\]

\begin{lemma}\label{A3=0}
If $(\ggo_{A,B,X,\ngo},J,\ip)$ is SKT, then $A_{\zg}=0$ and $B_{\zg}=0$. %In particular, if $A, B \neq 0$, then  $\zg$ coincides with the center of $\ggo_{A,B,X,\ngo}$. 
\end{lemma}

\begin{proof} According to \eqref{eqn_dc}, for $Z\in\zg$ we have that
\begin{align*}
&dc(e_1,e_2,Z,JZ)=\\
&=-c(X,Z,JZ)+c(AZ,e_2,JZ)-c(AJZ,e_2,Z)-c(BZ,e_1,JZ)+c(BJZ,e_1,Z)\\
&=\la (JAJA+AJAJ+JBJB+BJBJ)Z,Z\ra- |AZ|^2- |AJZ|^2-|BZ|^2 -|BJZ|^2\\
&=\la (J[A,B]+[A,B]J+2(BJA-AJB-A^2-B^2))Z,Z\ra- |AZ|^2- |AJZ|^2-|BZ|^2 -|BJZ|^2.
\end{align*}
The last equation follows from Remark \ref{ABJ}. On the other hand, since $[A,B]=\ad(X)$, for all $Z\in\zg$ we have that $[A,B]Z$ and $[A,B]JZ$ vanish. Hence, the SKT condition yields
\begin{equation}\label{SKT-eq1}
0=\la (BJA-AJB)Z,Z\ra- \la A^2Z,Z\ra-\la B^2Z,Z\ra - \frac{1}{2}(|AZ|^2+ |AJZ|^2+|BZ|^2 +|BJZ|^2),
\end{equation}
for every $Z\in\zg$. If we sum over any orthonormal basis of $\zg$, then \eqref{SKT-eq1} gives us
\begin{equation}\label{SKT-eq2}
0=\tr(B_{\zg}J_{\zg} A_{\zg}-A_{\zg}J_{\zg} B_{\zg})-\tr{A_{\zg}^2}-\tr{B_{\zg}^2}-|A_{\zg}|^2-|B_{\zg}|^2. \end{equation}
By the Jacobi condition, $A_{\zg}$ and $B_{\zg}$ commute and since they are nilpotent, \eqref{SKT-eq2} implies that
$$
0=|A_{\zg}|^2+|B_{\zg}|^2,$$ 
then $A_{\zg}=B_{\zg}=0.$
\end{proof}

\begin{corollary}\label{A=0-abelian}
If $(\ggo_{A,B,X,\ngo},J,\ip)$ is SKT and $\ngo$ is abelian, then $\ggo_{A,B,X,\ngo}$ is at most $2$-step nilpotent.
\end{corollary}

\begin{proof}
The proof follows immediately from Lemma \ref{A3=0} since $\ngo=\zg$ when it is abelian. 
\end{proof}

\begin{remark}
It follows from $A,B\in\Der(\ngo)$, $A_{\zg}=B_{\zg}=0$ and $[\ngo,\ngo]\subset\zg$, that 
\begin{equation}\label{skew} 
[AY,Z]=-[Y,AZ],\qquad [BY,Z]=-[Y,BZ], \qquad \forall Y,Z \in\ngo.
\end{equation}
\end{remark}

\begin{lemma}\label{lema-AB}
For any $Y\in\vg$, $[AY,BY]=0$. 
\end{lemma} 

\begin{proof}
Given $Y\in\vg$, then $[AY,BY]=-[BAY,Y]$ from \eqref{skew}. 
Since $\ngo=\vg\oplus\zg$ and $A_{\zg}=B_{\zg}=0$, it follows that
\[
[BAY,Y]=[B_\vg A_\vg Y,Y].
\]
From the Jacobi condition, we know that $[A,B]_\vg=0$ and therefore $[A_\vg,B_\vg]=0$. Hence, 
\[
[AY,BY] = -[BAY,Y] = -[B_\vg A_\vg Y,Y] = -[A_\vg B_\vg Y,Y] = -[A B Y,Y] = [BY,AY] = -[AY,BY],
\]
and the assertion follows.\end{proof}

\begin{lemma}
If $(\ggo_{A,B,X,\ngo},J,\ip)$ is SKT, then $A_\vg=B_\vg=0$ and $X\in\zg$.
\end{lemma}

\begin{proof} Since $A_\vg$ and $B_\vg$ are nilpotent and commute, we can take  $Y\in\vg$ such that $A_\vg Y = B_\vg Y = 0$,  or equivalently,  $A Y\in\zg$ and  $B Y\in\zg$.  We now proceed by showing that $JY$ satisfies the same conditions. Recall that from Lemma \ref{centro}, it is sufficient to prove that $0=[AJY,JAJY]=[BJY,JBJY]$. By Remark \ref{ABJ},
\begin{equation}\label{integ}
    [AJY,JAJY]=[AJY,(BJ-JB-A)Y]=[AJY,BJY]-[AJY,JBY]-[AJY,AY],
\end{equation}
which vanish by Lemma \ref{lema-AB} and Proposition \ref{centro-Jinv} applied to $(\ngo,J_\ngo)$. 
Hence, $A_\vg JY=0$, and it analogously follows that $B_\vg JY=0$.

Furthermore, setting $\bg:=\Ker(A_\vg)\cap\Ker(B_\vg)\neq 0$, we showed that $\bg$ is $J$-invariant.
If we prove that $\bg=\vg$, the assertion follows. 

On the contrary, suppose that $\ag:=(\Ker(A_\vg)\cap\Ker(B_\vg))^\perp\neq\{0\}$ and  $A_\ag$ and $B_\ag$ are defined according to Notation \ref{projection}. It can be easily seen that $A_\ag$ and $B_\ag$ are nilpotent and commute, therefore, there exists $0\neq W\in\ag$ such that $A_\ag W=B_\ag W=0$. In other words, $AW,BW\in \bg \oplus\zg$. 

In the same way as we proceed after equation  \eqref{integ}, we can show that 
\[[AJW,JAJW]=[JW,AJBW]+[JW,A^2W].\]
From the fact that $A(\bg\oplus\zg)\subseteq\zg$ and $\bg\oplus\zg$ is $J$-invariant, it follows that $AJW\in\zg$. In the same manner we can prove that $BJW\in\zg$. Therefore, $JW\in\bg$ which leads to a contradiction since $\bg$ is $J$-invariant and $W\in\ag$. We conclude that $A_\vg=B_\vg=0$, so
\[
A=\left[\begin{matrix}0& 0\\ \ast & 0 \end{matrix}\right], \qquad B=\left[\begin{matrix}0& 0\\ \ast & 0 \end{matrix}\right].
\]
The fact that $X$ lies in $\zg$ follows immediately from the Jacobi condition, i.e. $[A,B]=\ad(X)$. 
\end{proof}

An immediate consequence of the Lie algebra structure of $\ggo_{A,B,X,\ngo}$ given in \eqref{structure} and the above lemma is the following result. 

\begin{corollary}\label{n-2p}
If $(\ggo_{A,B,X,\ngo},J,\ip)$ is SKT, then $\ggo_{A,B,X,\ngo}$ is at most $2$-step nilpotent.
\end{corollary}

\subsection{General case} In the above sections we proved Theorem \ref{main theorem} for two particular cases. We are now in position to prove Theorem \ref{main theorem} in the general case, which is the main result of this article.

\begin{theorem}
If $(\ggo,J,\ip)$ is SKT with $\ggo$ nilpotent, then $\ggo$ is at most $2$-step nilpotent.
\end{theorem}

\begin{proof}
The proof is by induction on $n$, where $\dim\ggo=2n$. It is clear that the assertion is true for $n=1$. Suppose that it holds for every SKT nilpotent Lie algebra of dimension $2k$,  with $k<n$.

By the discussion at the beginning of Section \ref{sec-conj} there exists $A,B\in \glg(2(n-1),\RR)$, $X\in\RR^{2(n-1)}$ and $\ngo$ ideal of $\ggo$ of dimension $2(n-1)$ such that 
\[
\ggo=\ggo_{A,B,X,\ngo}.
\]
By Proposition \ref{ideal}, $(\ngo,J_\ngo,\ip|_\ngo)$ is SKT and of course nilpotent. Then, by hypothesis, $\ngo$ is at most $2$-step nilpotent. 
We are now under the hypothesis of Corollary \ref{A=0-abelian} or Corollary \ref{n-2p}, and this implies that $\ggo$ is at most $2$-step nilpotent.\end{proof}

\section{Construction of examples}\label{contructionofexamples} 

In this section we present a method to construct examples of SKT Lie algebras of arbitrary dimensions. The idea is to start with two SKT Lie algebras of dimension $n_1$ and $n_2$ that satisfy certain condition, and to construct a new SKT Lie algebra of dimension $n_1+n_2+2$. With this method and some already known examples, we can provide an example of an SKT Lie algebra of any even-dimension.

\subsection{A new construction}\label{construction}

For $i=1,2$, let $(\ngo_i,J_i,\ip_i)$ be an {\it irreducible} $2$-step nilpotent SKT Lie algebra. It is to say, $\ngo_i$ can not be decomposed as an orthogonal sum of $J$-invariant ideals, or equivalently, it is not a product of SKT Lie algebras of lower dimensions (see Remark \ref{ideal-coro}). Suppose in addition that for each $i=1,2$, 
$$\ngo_i=\vg_i\oplus \zg_i\quad \mbox{ and} \quad \dim{\zg_i}>\dim{[\ngo_i,\ngo_i]}.$$
Set $n_i:=\dim{\ngo_i}$, for $i=1,2$, and let $\{X_1,\dots,X_{n_1}\}$ and  $\{Y_1,\dots,Y_{n_2}\}$ be orthonormal basis of $(\ngo_1,\ip_1)$ and $(\ngo_2,\ip_2)$, respectively. There is no loss of generality in assuming that 
\[X_{n_1}\in\zg_1\cap[\ngo_1,\ngo_1]^\perp, \quad Y_{n_2}\in\zg_2\cap[\ngo_2,\ngo_2]^\perp.\]

Let $\ggo$ be the Lie algebra with underlying vector space $\ngo_1\oplus\ngo_2\oplus\RR^2$. Take $Z,W\in\ggo$ such that  $\{X_1,\dots,X_{n_1},Y_1,\dots,Y_{n_2},Z,W\}$ is a basis of $\ggo$, and consider $\ip$ which makes it an orthonormal basis. It is obvious that $\ip|_{\ngo_1\times\ngo_1}=\ip_1$ and  $\ip|_{\ngo_2\times\ngo_2}=\ip_2$. 

Let the Lie bracket on $\ggo$ be determined by
\begin{equation}\label{corchete}
\lb|_{\ngo_1\times\ngo_1}=\lb_{\ngo_1},\quad \lb|_{\ngo_2\times\ngo_2}=\lb_{\ngo_2},\quad [Z,W]=X_{n_1}+Y_{n_2},    
\end{equation}
and the complex structure defined as
$$J=\left[\begin{smallmatrix}\\J_1&&&\\&J_2&&\\&&0&-1\\&&1&0\end{smallmatrix}\right].$$
In particular $JZ=W$ and $[Z,JZ]\in([\ngo_1,\ngo_1]\oplus[\ngo_2,\ngo_2])^\perp$. 

In order to prove that $(\ggo,J,\ip)$ is SKT, we only have to check that
$$0=dc(Z,JZ,W_1,W_2)=dc(Z,W_1,W_2,W_3)=dc(JZ,W_1,W_2,W_3),\quad W_1,W_2,W_3\in\ngo_1\cup\ngo_2.$$
Indeed, by \eqref{eqn_dc}
\begin{align*}
dc(Z,JZ,W_1,W_2)=&\quad\la[J[Z,JZ],JW_1],W_2\ra+\la[JW_1,JW_2],[Z,JZ]\ra+\la[JW_2,J[Z,JZ]],W_1\ra \\
     &+\la [J[W_1,W_2],JZ],JZ\ra+\la [JZ,JJZ],[W_1,W_2]\ra+\la [JJZ,J[W_1,W_2]],Z\ra,
\end{align*}
which vanishes because $J$ preserve $\zg_1\oplus\zg_2$, $\ggo$ is 2-step nilpotent and $[Z,JZ]$  is orthogonal to $[\ngo_1,\ngo_1]\oplus[\ngo_2,\ngo_2]$. On the other hand, it is immediate from \eqref{eqn_dc} and the Lie algebra structure of $\ggo$, that $dc(Z,W_1,W_2,W_3)$ and $dc(JZ,W_1,W_2,W_3)$ vanish.

It only remains to see that $(\ggo,J,\ip)$ is irreducible. Suppose that there exists an orthogonal $J$-invariant decomposition of ideals $$\ggo=\ag\oplus\bg,$$ 
where $\ag$ is irreducible. If $\ngo_1\cap\ag\neq\{0\}$, then it is a $J$-invariant ideal contained in $\ag$ irreducible. Therefore, $\ag=\ngo_1$ and $\bg=\ngo_2\oplus\vspan\{Z,JZ\}$. This contradicts the fact that $\bg$ is an ideal, since $[Z,JZ]=X_{n_1}+Y_{n_2}\in\ngo_1\oplus\ngo_2$. If $\ngo_2\cap\ag\neq\{0\}$, we can proceed analogously and to get a contradiction. Finally, if $\ngo_i\cap\ag=\{0\}$ for $i=1,2$, it follows that $(\ngo_1\oplus\ngo_2) \cap \ag= \{0\}$ by using that it is an ideal of $\ag$ and $\ag$ is irreducible. Then $\ag$ has to be zero. Indeed, if $A\in\ag$, $A=N_1+N_2+\alpha Z + \beta JZ,$ with $\alpha \neq 0$ or $\beta \neq 0$. Then, $[A,Z] \in \ag$ and $[A,JZ] \in \ag$, which means that $X_{n_1} + Y_{n_2} \in \ag$, and we obtain a contradiction.

%On the contrary, if $\ngo_i\cap\ag=\{0\}$ for $i=1,2$, then $\ag=\vspan\{Z,JZ\}$ because $\ag$ is $J$-invariant, which implies that $\ag$ is not an ideal and leads to a contradiction. Therefore, $\ggo$ must be irreducible.

\begin{remark}\label{repeat}
Observe that a  quick computation shows that the SKT Lie algebra $(\ggo, J, \ip)$ obtained by the above construction satisfies $\dim \zg(\ggo)>\dim [\ggo,\ggo]$. To the obtained example, we can apply the construction again in order to get higher dimensional examples.
\end{remark}

\begin{remark}
 Setting $[Z,W]=rX_{n_1}+sY_{n_2}$, for $s,t\in\RR-\{0\}$, instead of $[Z,W]=X_{n_1}+Y_{n_2}$ in \eqref{corchete}, we obtain a family of examples of SKT Lie algebras. An interesting question is whether they are pairwise non-equivalent.
\end{remark}

\begin{remark}\label{suma-ab}
In the previous construction, if both $J_1$ and $J_2$ are abelian, then $J$ results abelian, and if one of them is not, then $J$ is not abelian. Recall that a complex structure $J$ on $\ggo$ is called \emph{abelian} if $[JX,JY]=[X,Y]$ for all $X,Y \in \ggo$.
\end{remark}

%\subsection{Known examples}
\subsection{Known examples}\label{known examples} In this section, we present some known examples of SKT Lie algebras to set up some notation.

\begin{example}\label{ej_dim4}\cite{MadsenSwann} Consider the $4$-dimensional Lie algebra $\ngo_1$ with basis $\{e_1,\dots,e_4\}$ and Lie bracket determined by
\[de^3=-e^{12}.\] 
Let $\ip_1$ be the inner product such that the basis is orthonormal, and the abelian complex structure $J_1$ is defined by,
\[J_1e_1=e_2,\quad J_1e_3=e_4.\]
The Hermitian manifold $(\ngo_1,J_1,\ip_1)$ has the following torsion $3$-form of the Bismut connection
\[c=-e^{123},\] 
which turns out to be closed and therefore $(\ngo_1,J_1,\ip_1)$ is an SKT Lie algebra. Note that if $\zg_1$ is the center of $\ngo_1$, then $\zg_1\cap[\ngo_1,\ngo_1]^\perp=\vspan\{e_4\}$.
\end{example}

\begin{example}\cite{FinoPartonSalamon,Ugarte07}\label{ej_dim6} Let $\ngo_2$ be the $6$-dimensional Lie algebra with basis $\{f_1,\dots,f_6\}$ and Lie bracket determined by
\[df^5=-f^{12}+f^{14}-f^{23}-f^{34}.\] 
Let $\ip_2$ be the inner product such that the basis is orthonormal, and the abelian complex structure $J_2$ is defined by,
\[J_2f_1=f_2,\quad J_2f_3=f_4,\quad J_2f_5=f_6.\]
Then, the torsion $3$-form of the Bismut connection of $(\ngo_2,J_2,\ip_2)$ is
\[c=-f^{125}+f^{145}-f^{235}-f^{345},\] 
and it is closed, so $(\ngo_2,J_2,\ip_2)$ is SKT.
Observe that if $\zg_2$ is the center of $\ngo_2$, then $\zg_2\cap[\ngo_2,\ngo_2]^\perp=\vspan\{f_6\}$.
\end{example}

\begin{example}\cite{EFV12}\label{ej_dim8} Consider the $8$-dimensional Lie algebra $\ngo_3$ with basis $\{v_1,\dots,v_8\}$ and Lie bracket determined by
\[dv^5=-2v^{12}+v^{14}-v^{34},\quad dv^6=-v^{13},\quad dv^7=-v^{12}+v^{34}.\] 
Let $\ip_3$ be the inner product such that the basis is orthonormal, and the non-abelian complex structure $J_3$ defined by,
\[J_3v_1=v_2,\quad J_3v_3=v_4,\quad J_3v_5=v_6,\quad J_3v_7=v_8.\]
The Hermitian manifold $(\ngo_3,J_3,\ip_3)$ has the following torsion $3$-form of the Bismut connection
\[c=-2v^{125}-v^{127}-v^{235}-v^{246}-v^{345}+v^{347},\] 
which is closed, and therefore $(\ngo_3,J_3,\ip_3)$ is SKT.
Note that if $\zg_3$ is the center of $\ngo_3$, then $\zg_3\cap[\ngo_3,\ngo_3]^\perp=\vspan\{v_8\}$.
\end{example}

\subsection{Applications}\label{aplications}
The aim of this section is to apply the construction given in Section \ref{construction}. We provide two new examples of SKT Lie algebras by using Examples \ref{ej_dim4}, \ref{ej_dim6} and \ref{ej_dim8}.

\begin{example}
Let $(\ngo_1,J_1,\ip_1)$ and $(\ngo_2,J_2,\ip_2)$ be the irreducible SKT Lie algebras defined in Examples \ref{ej_dim4} and \ref{ej_dim6}, respectively. According to the method presented in Section \ref{construction}, we can construct a $(4+6+2)$-dimensional SKT Lie algebra $\ggo$ with orthonormal basis $\{e_1,\dots,e_4,f_1,\dots,f_6,w_1,w_2\}$, Lie bracket determined by,
\[de^3=-e^{12},\quad df^5=-f^{12}+f^{14}-f^{23}-f^{34},\quad de^4=-w^{12},\quad df^6=-w^{12}\] and complex structure
\[J|_{\ngo_1}=J_1,\quad J|_{\ngo_2}=J_2,\quad Jw_1=w_2. \]
Indeed, the resulting torsion $3$-form of the Bismut connection is
\[c=-e^{123}-2v^{125}-v^{127}-v^{235}-v^{246}-v^{345}+v^{347}-e^4\wedge w^{12}-f^6\wedge w^{12},\]    
which is closed and therefore $(\ggo,J,\ip)$ is SKT. An easy computation shows that $J$ is abelian, which is consistent with Remark \ref{suma-ab}.
\end{example}

%\begin{remark}
%We can obtain a $2$-parameter family of SKT Lie algebras setting 
%\[[w_1,w_2]= re_4+sf_6,\;r,s\neq 0\quad\mbox{ instead of }\quad [w_1,w_2]=e_4+e_6.\]
%\end{remark}

\begin{example}
Let $(\ngo_1,J_1,\ip_1)$ and $(\ngo_3,J_3,\ip_3)$ be the irreducible SKT Lie algebras defined in Examples \ref{ej_dim4} and \ref{ej_dim8}, respectively. As we did in the previous example, we construct a $(4+8+2)$-dimensional SKT Lie algebra $\ggo$ with orthonormal basis $\{e_1,\dots,e_4,v_1,\dots,v_8, w_1,w_2\}$, Lie bracket determined by,
\[\begin{array}{c}
    de^3=-e^{12},\quad dv^5=-2v^{12}+v^{14}-v^{34},\quad dv^6=-v^{13},\quad \\
    dv^7=-v^{12}+v^{34},\quad
    de^4=-w^{12},\quad dv^8=-w^{12},
\end{array}\] and complex structure:
\[ J|_{\ngo_1}=J_1,\quad J|_{\ngo_3}=J_3,\quad Jw_1=w_2.\]
Indeed, the resulting torsion $3$-form of the Bismut connection is
\[c=-e^{123}-2v^{125}-v^{127}-v^{235}-v^{246}-v^{345}+v^{347}-e^4\wedge w^{12}-v^8\wedge w^{12},\]
which is closed and therefore $(\ggo,J,\ip)$ is SKT. Note that $J$ is not abelian since $J_3$ is not abelian.

\end{example}

\begin{remark}
It is worth pointing out that with Examples \ref{ej_dim4}, \ref{ej_dim6} and \ref{ej_dim8} it can be reached at least one example of an irreducible SKT Lie algebra on any even-dimension by applying the construction repeatedly (see Remark \ref{repeat}). For instance, in order to obtain an example of dimension $4+6m,$ with $m\in\NN$, the only needed SKT Lie algebra is $(\ngo_1,J_1,\ip_1)$ of Example \ref{ej_dim4}. In fact, applying the construction to $(\ngo_1,J_1,\ip_1)$ and $(\ngo_1,J_1,\ip_1)$, an SKT Lie algebra of dimension $10$ is obtained. Using the new SKT Lie algebra and again $(\ngo_1,J_1,\ip_1)$, an SKT Lie algebra of dimension $16$ is constructed, and go on. Analogously, examples of dimensions $6+6m$ and $8+6m,$ with $m\in\NN$, can be obtained from $(\ngo_1,J_1,\ip_1)$ and $(\ngo_2,J_2,\ip_2)$ given in Examples \ref{ej_dim4} and \ref{ej_dim6}, and $(\ngo_1,J_1,\ip_1)$ and $(\ngo_3,J_3,\ip_3)$ given in Examples \ref{ej_dim4} and \ref{ej_dim8}, respectively.
\end{remark}

\subsection{More examples of SKT Lie algebras with non-abelian complex structures}

\begin{example}\cite{FinoPartonSalamon}\cite{Ugarte07}\label{ej_dim6_nonabelian}
Let $\ngo$ be the $6$-dimensional Lie algebra with basis $\{e_1,\dots,e_{6}\}$ and Lie bracket determined by
\[de^5=-e^{12}-e^{14}-e^{34},\quad de^6=e^{13}.\]
Let $\ip$ be the inner product such that the basis is orthonormal, and the non-abelian complex structure $J$ defined by,
\[Je_1=e_2,\quad Je_3=e_4,\quad Je_5=e_6.\]
The torsion $3$-form of the Bismut connection of $(\ngo,J,\ip)$ is
\[c=-e^{125}+e^{235}+e^{246}-e^{345},\] 
and it is closed. Therefore $(\ngo,J,\ip)$ is SKT.
Note that if $\zg$ is the center of $\ngo$, then $\zg=[\ngo,\ngo]=\vspan\{e_5,e_6\}$.
\end{example}

\begin{example}\label{ej_dim10} Consider the $10$-dimensional Lie algebra $\ngo$ with basis $\{e_1,\dots,e_{10}\}$ and Lie bracket determined by
{\small
\[de^7=-e^{12}+e^{24}-e^{34}-2e^{36},\quad de^8=-e^{14}-\tfrac{5}{2}e^{34}+2e^{35}-2e^{56},\quad de^9=-e^{12}+e^{16}-e^{25}+e^{36}-e^{45}-e^{56}.\]}
Let $\ip$ be the inner product such that the basis is orthonormal, and the non-abelian complex structure $J$ defined by,
\[Je_{2i-1}=e_{2i},\quad \forall i\in\{1,\dots,5\}.\]
The Hermitian manifold $(\ngo,J,\ip)$ has the following torsion $3$-form of the Bismut connection
\[c=-e^{127}-e^{129}+e^{137}+e^{169}+e^{238}-e^{259}-e^{347}-\tfrac{5}{2}e^{348}+e^{369}+2e^{457}-e^{459}+2e^{469}-2e^{568}-e^{569},\] 
which turns out to be closed and therefore $(\ngo,J,\ip)$ is SKT.
Observe that if $\zg$ is the center of $\ngo$, then $\zg\cap[\ngo,\ngo]^\perp=\vspan\{e_{10}\}$.
\end{example}

\begin{example}\label{ej_dim12} Consider the $12$-dimensional Lie algebra $\ngo$ with basis $\{e_1,\dots,e_{12}\}$ and Lie bracket determined by
\[\begin{array}{c}de^7=-e^{12}+e^{24},\quad de^8=-e^{14}+2e^{16}-2e^{25},\quad de^9=-e^{12}-e^{34}-e^{56},\\ de^{10}=-e^{34},\quad de^{11}=-e^{12}+e^{36}-e^{45}-3e^{56}.\end{array}\]
Let $\ip$ be the inner product such that the basis is orthonormal, and the non-abelian complex structure $J$ defined by,
\[Je_{2i-1}=e_{2i},\quad \forall i\in\{1,\dots,6\}.\]
The the Bismut connection of $(\ngo,J,\ip)$ has the following torsion $3$-form  
\[c=-e^{127}-e^{129}-e^{12\,11}+e^{137}+2e^{168}+e^{238}-2e^{258}-e^{349}-e^{34\,10}+e^{36\,11}-e^{45\,11}-e^{569}-3e^{56\,11},\] 
which is closed, and so $(\ngo,J,\ip)$ is SKT.
Note that if $\zg$ is the center of $\ngo$, then $\zg\cap[\ngo,\ngo]^\perp=\vspan\{e_{12}\}$.
\end{example}

\begin{proposition}
For each natural $n\geq 3$, there exists at least one $2n$-dimensional SKT Lie algebra with non-abelian complex structure. 
\end{proposition}

\begin{proof}
For $n=3,4,5,6$, see Examples \ref{ej_dim6_nonabelian}, \ref{ej_dim8}, \ref{ej_dim10} and  \ref{ej_dim12}. In order to obtain examples of higher dimensions, the construction described above can be repeatedly applied, starting with one SKT Lie algebra with $J$ non-abelian. For instance, in order to obtain an example of dimension $14$, the construction can be applied to the SKT Lie algebras  $(\ngo_1,J_1,\ip_1)$ and $(\ngo_3,J_3,\ip_3)$ from Examples \ref{ej_dim4} and \ref{ej_dim8}, respectively. Then, using the new SKT Lie algebra and Example \ref{ej_dim4}, a new SKT Lie algebra of dimension 20 is obtained, and with an inductive argument, SKT Lie algebras of dimension $8+6m,$ with $m\in\NN$ are reached. Analogously, examples of dimensions $10 + 6m$ and $12 + 6m$, with $m \in \NN$, are
obtained from  Examples \ref{ej_dim4} and \ref{ej_dim10}, and Examples \ref{ej_dim4} and \ref{ej_dim12}, respectively.
\end{proof}

\bibliography{skt}
\bibliographystyle{amsalpha}
\end{document}